\documentclass{article}

\usepackage{arxiv}

\usepackage[utf8]{inputenc} 	
\usepackage[T1]{fontenc}    	
\usepackage{hyperref}       	
\usepackage{url}           		
\usepackage{booktabs}       	
\usepackage{amsfonts}       	
\usepackage{nicefrac}       	
\usepackage{microtype}      	
\usepackage{lipsum}		
\usepackage{graphicx}

\usepackage[mode=buildmissing,subpreambles=true]{standalone}
\usepackage{subcaption}

\usepackage{enumitem}

\usepackage{tikz}
\usepackage{amsmath,amsthm} 		
	\theoremstyle{definition}
	\newtheorem{definition}{Definition}[section]
	\newtheorem{example}{Example}[section]
	\newtheorem{theorem}{Theorem}[section]
	\newtheorem{coro}{Corollary}[section]
	\newtheorem{prop}{Proposition}[section]

\usepackage{pgfplots}
\pgfplotsset{compat=1.17}
\usetikzlibrary{calc}

\title{Quadratic Formula: Revisiting a Proof\\ Through the Lens of Transformations}

\author{Shawn S.~Wirts
	\\
	Department of Mathematics\\
	Fresno Pacific University\\
	Fresno, CA 93702\\
	\texttt{shawn.wirts@fresno.edu} \\
}

\renewcommand{\shorttitle}{Quadratic Formula: Revisiting a Proof Through the Lens of Transformations}

\hypersetup{
pdftitle={Quadratic Formula: Revisiting a Proof Through the Lens of Transformations}
pdfsubject={math.GM},
pdfauthor={Shawn S.~Wirts},
pdfkeywords={Quadratic Formula, Proof, Transformation, Quadratics},
}

\pgfplotsset{compat=1.17} 
\begin{document}
\maketitle

\begin{abstract} 
	In this paper,
	we derive the quadratic formula
	as a consequence of constructively proving
	the existence of standard and factored forms
	for general form real quadratic functions.
	Emphasis is put on connections to graphing of corresponding parabolas
	through function transformation perspectives of scaling and translation.
\end{abstract}


\section{Introduction}

In the routine of teaching college mathematics
at an introductory and general education level,
frequent reconsideration of
how best to present certain topics
is a common challenge.
Derivations of significant results certainly have their place,
but tradition and routine can often create
some intellectual stagnancy,
and a shift in perspective
can be refreshing for instructors and students.

Po-Shen Loh's expositions \cite{loh2019simple,loh2020explanation}
on deriving the quadratic formula through
factorization and symmetry arguments
is a less common algebraic approach
than a traditional completing-the-square method.
In 2019, it caught some attention in mathematics education circles,
likely due to its refreshing derivation method.
As Loh points out,
the approach is hardly new
in the sense of the history of algebra,
but certainly may have fallen out of favor
in modern western mathematics education.

Perhaps the most strikingly satisfying features
of the method Loh revisits to solve a quadratic equation
is the symmetry of the set-up using linear factorization
blended with the connection to an ancient mathematics problem:
solving for two unknowns given their known product and sum.
Though this approach has its own form of algebraic elegance,
there are some that may find the symbolic algebraic arguments
only marginally more accessible
than a traditional completing-the-square based derivation
of the quadratic formula.

In some presentations of college algebra and precalculus level content
(see Blitzer's \emph{Precalculus} \cite{RBlitzer}),
analysis of quadratic functions
and derivation of the quadratic formula
follows development of both
the topics of function transformation
and a basic introduction to the complex plane $\mathbb{C}$
and its component-based algebra.
For those craving a visually accessible motivation
for the linear factorization based derivation,
as well as to simultaneously build connections between
general, standard, and factored forms of quadratics,
we revisit a derivation of the quadratic formula
inspired by Loh's perspective
and additionally under an assumption
of a familiarity with function transformation.

\subsection{A Foundation of Transformations}
\label{sec:trans}

At the Precalculus level in the development of introductory function theory in $\mathbb{R},$
one encounters the concept of relating
basic ``parent'' functions and their Cartesian graphs
to the family of transformed functions and graphs
through scaling (colloquially stretching, squishing, and reflecting)
and translation (shifting).

Given a parent function $g(t)$
with translation constants $C,D\in\mathbb{R}$
and non-zero scaling constants $A,B\in\mathbb{R},$
a transformed version of function $g$ into a function $f$ may be given by
\begin{align*}
	f(x)&=A\cdot g\bigl(B\cdot x+C\bigr)+D.
\end{align*}
Using the domain and range variable correspondence system
\begin{equation}  \label{eq:corsystem}
	\begin{cases}
		\hfill t & =\ B\cdot x +C\\
		A\cdot s+D& =\ y
	\end{cases}
\end{equation}
to connect parent-graph data to transformed-graph data,
a correspondence of
points $(t,s)$ on the graph of $s=g(t)$
with
points $(x,y)$ on the graph of $y=f(x)$
through the point correspondence of
\begin{align}  \label{eq:correspondence}
	\Bigl(t,s\Bigr)
	&\rightarrow \Bigl(x,y\Bigr)
	=\left(\frac{t}{B}-\frac{C}{B}\ ,\ A\cdot s +D\right)
	=\left(\frac{t-C}{B}\ ,\ A\cdot g(t) +D\right)
	=\Bigl(x,f(x)\Bigr)
\end{align}
may be achieved.

Avoiding reuse of traditional $x$ and $y$ variables for both
parent-function input  and output
as well as transformed-function input and output
allows for some simplicity in using substitutions to relate data
on the corresponding graphs,
as well as more readily adapt to tabular functions' data
in the abstraction of transformations.

\begin{example}\label{ex:wave}
	Suppose it is given that parent wave function
	$g:\mathbb{R}\rightarrow\mathbb{R}$
	given by $g(t)=\sin(t)$
	includes points $\left(-\frac{\pi}{2},-1\right),\left(0,0\right),\left(\frac{\pi}{2},1\right)$
	on the graph $s=g(t)$ in the $ts-$plane,
	and as well has a graph that exhibits rotational symmetry about the origin
	due to $g$ being an odd function.
	Consider the transformed wave function given by
	\begin{align*}
		f(x)=3\cdot g\left(\pi\cdot x+\frac{\pi}{2}\right)+2.
	\end{align*}
	
	The data correspondence system
	\begin{equation*} 
		\begin{cases}
			\hfill t & =\ \pi\cdot x +\frac{\pi}{2}\\
			3\cdot s+2& =\ y
		\end{cases}
	\end{equation*}
	implies that by indexing the correspondence between data points we can calculate
	\begin{align*}
		\Bigl(t_k,s_k\Bigr)
		&\rightarrow
		\Bigl(x_k,y_k\Bigr)
		& \text{ where }
		\Bigl(x_k,y_k\Bigr)
		&=\left(\frac{t_k}{\pi}-\frac{1}{2}\ ,\ 3\cdot s_k +2\right), \text{ hence }
		\\
		\Bigl(t_1,s_1\Bigr)
		=\left(-\frac{\pi}{2},-1\right)
		&\rightarrow
		\left(-1\ ,\ -1\right)
		&\text{ as }
		\Bigl(x_1,y_1\Bigr)
		&=\left(\frac{(-\pi/2)}{\pi}-\frac{1}{2}\ ,\ 3\cdot (-1) +2\right),
		\\
		\Bigl(t_2,s_2\Bigr)
		=\left(0,0\right)
		&\rightarrow
		\left(\frac{1}{2}\ ,\ 2\right)
		&\text{ as }
		\Bigl(x_2,y_2\Bigr)
		&=\left(\frac{(0)}{\pi}-\frac{1}{2}\ ,\ 3\cdot (0) +2\right),
		\\
		\text{ and }
		\Bigl(t_3,s_3\Bigr)
		=\left(\frac{\pi}{2},1\right)
		&\rightarrow 
		\left(0\ ,\ 5\right)
		&\text{ as }
		\Bigl(x_3,y_3\Bigr)
		&=\left(\frac{(\pi/2)}{\pi}-\frac{1}{2}\ ,\ 3\cdot (1) +2\right).
	\end{align*}
	As an example of a transformed wave function,
	the graph of $y=3\cdot \sin\left(\pi\cdot x+\frac{\pi}{2}\right)+2$
	is readily viewed having its the parent wave-function's
	amplitude vertically-scaled by multiplying by $3,$
	its equilibrium vertically-shifted up by $2,$
	its period rescaled to cycle lengths of $2$ (by horizontally-scaling through division by $\pi$),
	and introducing a phase-shift of $-\frac{1}{2}$ (by finally left-shifting the graph of the wave).
	This sort of analysis also conveniently identifies the transformed extrema points
	where a new minimum is achieved at $(-1,-1)$
	and a new maximum is achieved at $(0,5).$
	Additionally we identify a new point $\left(\frac{1}{2},2\right)$
	about which $y=f(x)$ has rotational symmetry.

	\begin{figure}[h]
		\centering
		\includegraphics
		  	{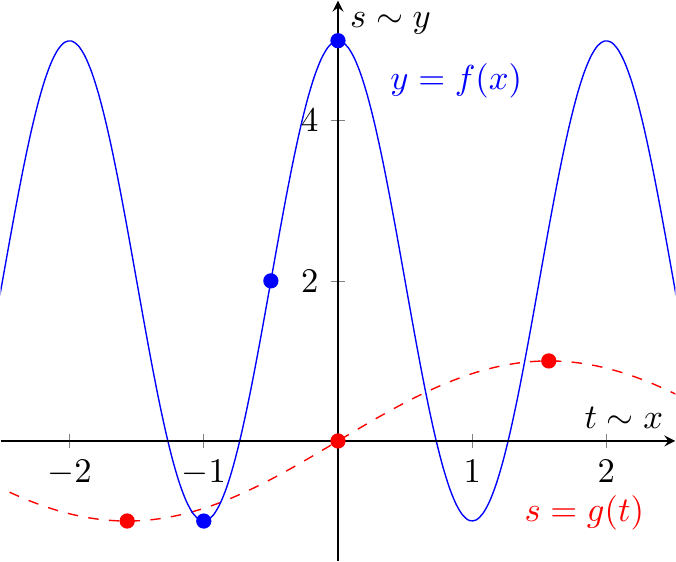}
		\caption{Transformed wave $y=f(x)$ superimposed over a parent wave $s=g(t)$ in Example \ref{ex:wave}.}
		\label{fig:waves}
	\end{figure}

\end{example}

The correspondence approach summarized in \ref{eq:correspondence}
also can help students to understand
the qualitative graphical interpretations of the ordered effects
of the constants used in the transformation.
The intuitive vertical-scaling of output data caused by multiplication by $A$
followed by vertical-shifting (up) by addition of $D$
can be paired with the often counter-intuitive effects
of horizontal-scaling by \emph{division} of input-data by $B$
followed by horizontal-shifting (right) by $-\frac{C}{B}.$
Alternatively, there is also the perspective of
first horizontally-shifting (right) by $-C$
followed by the horizontal-scaling through division by $B,$
in which the addition (shifting)
ends up taking order of operation priority 
over multiplication (scaling)
due to the inversion steps
needed to connect \ref{eq:corsystem} to \ref{eq:correspondence}.
Regardless of choice of the non-commutative order
of horizontal transformation effects,
substitution can be used
to establish point correspondence
and subsequently interpret graphical effects.

\subsubsection{Transformations of Parabolas}
\label{sec:trans:parabola}

Students initially become familiar
with the basic squaring function
$g:\mathbb{R}\rightarrow\mathbb{R}$
given by $g(t)=t^2$
and the graphical features of its representative parabola
$s=g(t)$ (i.e. $s=t^2)$ within the Cartesian $ts-$plane,
including its graph's reflective symmetry across the vertical axis
due to $g$ being an even function,
as well as the function's absolute (and relative) minimum
$0=\min\{g(t) | t\in\mathbb{R}\}=g(0).$
Following understanding these basics,
it is an appropriate extension to consider
the general real quadratic functions.

\begin{definition}[General form]
	For fixed constants $b,c\in\mathbb{R}$
	and nonzero $a\in\mathbb{R},$
	the function $f:\mathbb{R}\rightarrow\mathbb{R}$
	given by
	\begin{align} \label{eq:general}
		f(x)&=a\cdot x^2+b\cdot x+c
	\end{align}
	is a (real) \emph{quadratic} function
	written in \emph{general form.}
\end{definition}

\begin{definition}[Standard form]
	For a (real) quadratic function $f$ given by
	$f(x)=a\cdot x^2+b\cdot x+c,$
	if constants $h,k\in\mathbb{R}$
	satisfy
	\begin{align}  \label{eq:standard}
		f(x)=a\cdot(x-h)^2+k,
	\end{align}
	then this form
	is said to be the \emph{standard form} of $f.$
\end{definition}

What is not immediately obvious
in the introduction of quadratic functions
is that it is possible to write any real quadratic
in a unique standard form.
A common verification of this result
is to complete-the-square,
which includes much of the algebra needed
to also derive the quadratic formula.
Since we will use a different method,
at this point we we treat the existence of the standard form
as a proposition.

\begin{figure}[h]
	\centering
	\begin{subfigure}{.4\textwidth}
		  \centering
		  \includegraphics
		  	{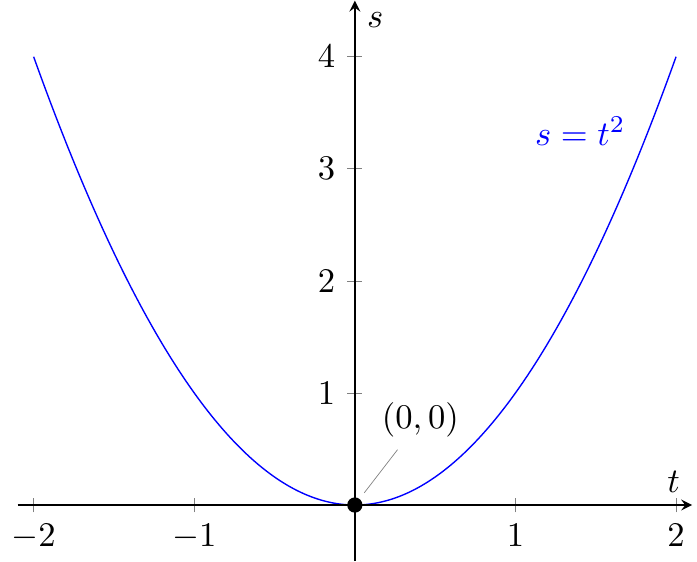}
		  \caption{Parent-function}
		  \label{fig:parent}
	\end{subfigure}
	\begin{subfigure}{.4\textwidth}
		  \centering
		  \includegraphics
		  	{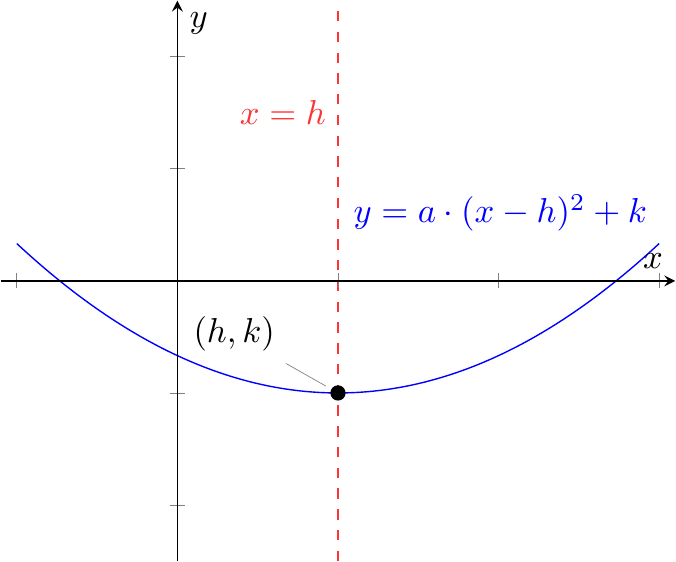}
		  \caption{Example transformed-function}
		  \label{fig:trans}
	\end{subfigure}
	\caption{Typical transformation of a parabolic graph (with $a>0$)}
	\label{fig:transparabola}
\end{figure}

Advantages of the standard form given in \ref{eq:standard}
include analysis of the behavior of the standard form quadratic
as a transformation of a basic squaring function,
as indicated in Figure \ref{fig:transparabola}.
This results in the identification of
horizontally-shifted reflective symmetry
in the graph of $y=f(x),$
about the axis-of-symmetry $x=h$
where we find the location of vertex $(h,k)$
(i.e. the point corresponding to the location of the minimum of the parent function $g$)
and the extreme value of the function $f$
of $k=f(h).$
Furthermore, the identification of the extremum of a quadratic
directly relates to the identification of the vertex of the function's parabolic graph.
Hence, the standard form of a quadratic function
is extremely useful in precalculus optimization applications
that rely on quadratic form models,
and delays the need for development of calculus-based optimization techniques.

\subsubsection{Transformations of Parabolas: Visualization of Cases for $x-$intercepts}
\label{sec:trans:cases}

Inspired by the transformed graph example in Figure \ref{fig:trans},
through consideration of variations of
vertical-scaling factor $a\in(-\infty,0)\cup(0,\infty)$
and
vertical-shifting factor $k\in\mathbb{R},$
and without loss of generality for choice of horizontal-shifting factor $h\in\mathbb{R},$
we produce results as seen in Figures \ref{fig:casesapos} and \ref{fig:casesaneg}.

With the reasonable assumption of a continuous dependence on parameters,
these cases clearly indicate the possibility of
two distinct $x-$intercepts (as in Figures \ref{fig:posanegk} and \ref{fig:negaposk}, and labeled as $r,s\in\mathbb{R}$),
one intercept, or none.
Of course, the expectation is that the study of the distinctions between these cases
will lead to development of the discriminant of the quadratic,
but at this point we merely use the visualization
to motivate the following insight:
\begin{quote}\label{insight}
	If the graph of a quadratic function written in standard form
	is to have distinct real $x-$intercepts ($r,s\in\mathbb{R}$),
	they must be equidistant from the axis-of-symmetry ($x=h$).
	So $h=\frac{r+s}{2}$ necessarily.
\end{quote}
It is this insight into the symmetry inherited through the transformation perspective
that we see paralleled in the purely algebraic approach to deriving the quadratic formula.
As we see play out in latter argumentation,
the convenience of identifying the distinct $x-$intercepts
of the graph as $r=h-z$ and $s=h+z$ for some $z\in\mathbb{C}$
will play a significant role in generalizing from the case of real quadratic roots
to the case of non-real quadratic roots.
Additionally, when continuously
parametrizing quadratic transformations
based on $z\in\mathbb{R}\cup i\mathbb{R},$
one can observe the transition between
cases such as in Figure \ref{fig:posanegk} when $z\in (0,\infty)$
to those in \ref{fig:posaposk} when  $z=i\xi$ for $\xi\in (0,\infty),$
with the critical case of \ref{fig:posazerok} for $z=0.$

\begin{figure}[h]
	\begin{subfigure}{.3\textwidth}
		  \centering
		  \includegraphics
		  	{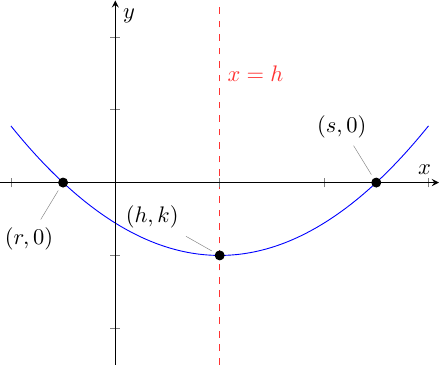}
		  \caption{$k<0$}
		  \label{fig:posanegk}
	\end{subfigure}
	\begin{subfigure}{.3\textwidth}
		  \centering
		  \includegraphics
		  	{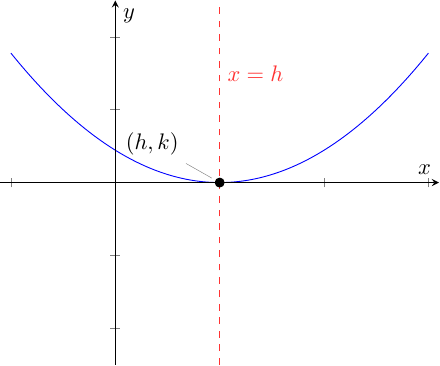}
		  \caption{$k=0$}
		  \label{fig:posazerok}
	\end{subfigure}
	\begin{subfigure}{.3\textwidth}
		  \centering
		  \includegraphics
		  	{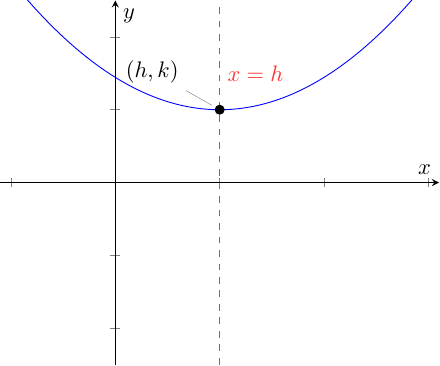}
		  \caption{$k>0$}
		  \label{fig:posaposk}
	\end{subfigure}
	\caption{Example cases of a transformed quadratic $f(x)=a\cdot(x-h)^2+k$ with $a>0$}
	\label{fig:casesapos}
\end{figure}

\begin{figure}[h]
	\begin{subfigure}{.3\textwidth}
		  \centering
		  \includegraphics
		  	{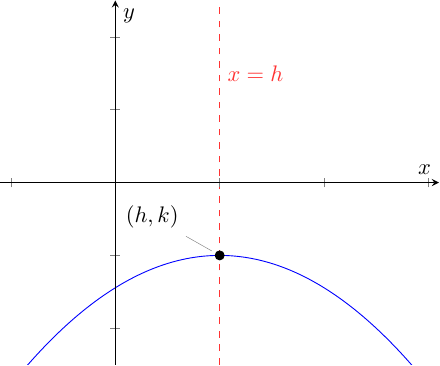}
		  \caption{$k<0$}
		  \label{fig:neganegk}
	\end{subfigure}
	\begin{subfigure}{.3\textwidth}
		  \centering
		  \includegraphics
		  	{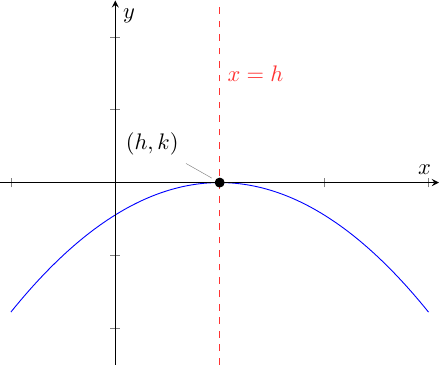}
		  \caption{$k=0$}
		  \label{fig:negazerok}
	\end{subfigure}
	\begin{subfigure}{.3\textwidth}
		  \centering
		  \includegraphics
		  	{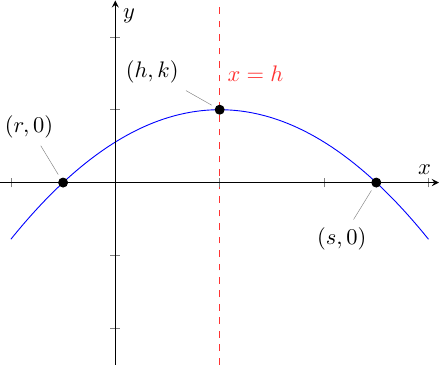}
		  \caption{$k>0$}
		  \label{fig:negaposk}
	\end{subfigure}
	\caption{Example cases of a transformed quadratic $f(x)=a\cdot(x-h)^2+k$ with $a<0$}
	\label{fig:casesaneg}
\end{figure}

\subsection{Towards Complex Linear Factorization of Real Quadratics}
\label{sec:factorization}

Having previously encountered the Zero Product Property
on $\mathbb{R}$ (and likely as well on $\mathbb{C}$),
one can infer from cases in Figures \ref{fig:posanegk},\ref{fig:posazerok},\ref{fig:negazerok}, and \ref{fig:negaposk}
that the existence of a standard form for a quadratic
(and hence the transformation perspective on the parabolic graph)
will necessarily imply the existence of a (real) factored form as well.

\begin{definition}[Linear factored form]
	For a (real) quadratic function $f$ given by
	$f(x)=a\cdot x^2+b\cdot x+c,$
	if constants $r,s\in\mathbb{C}$
	satisfy
	\begin{align}  \label{eq:factored}
		f(x)=a\cdot(x-r)(x-s),
	\end{align}
	then this form
	is said to be a \emph{(complex) linear factored form} of $f.$
\end{definition}

In some precalculus presentations (see \cite{RBlitzer}),
what becomes clear soon after the development
of the theory of real quadratics
is the value of complete linear factorization of polynomials
for purposes of solving polynomial equations
and graphing interpretations.
The Fundamental Theorem of Algebra
and corollaries support the goal
of pursuing linear factorization of polynomials
for such purposes,
at least in as much as an existence theorem can support.

Similar to our perspective on the standard form in \ref{eq:standard},
we treat the existence of the factored form
as a proposition to establish.
In fact, it is constructively proving
the existence of such roots $r,s\in\mathbb{C}$
that is equivalent to deriving the quadratic formula.
 
Unlike the unique representation of the standard form,
the factored form has some ambiguity built in
through choice of how to identify a pair of quadratic roots with $r$ and $s,$
in particular once the natural ordering of $\mathbb{R}$
is unable to be analogously applied when $r,s\in\mathbb{C}.$
Still, the standard result of 
complex polynomial factorization
being unique up to ordering still applies.

The non-constructive Fundamental Theorem of Algebra
requires consideration of potentially non-real roots of a real quadratic polynomial
prior to an approach based on appealing to the Factor Theorem 
to extract a complete linear factorization of general degree polynomials.
In deriving a quadratic formula for real quadratic functions,
one is forced to consider non-real roots at some point,
so it is convenient to have at least introduced $\mathbb{C}$
prior to computational derivation of the quadratic formula.

\subsection{A Classic System of Equations}
\label{sec:system}

\begin{example}
	Let $\alpha,\beta\in\mathbb{R}$ be given.
	Solve for $r,s$ such that
	\begin{equation} \label{eq:prodsum}
		\begin{cases}
			\alpha =r\cdot s
			\\
			\beta = r+s
		\end{cases}.
	\end{equation}
\end{example}
The statement of such a simple nonlinear system
hides within it a deeper philosophical prompting
than perhaps is apparent on the surface.
Even in the special case of $\beta=0,$ when $s=-r$ and $r^2=-\alpha,$
there is an imperative to develop
irrational solutions $r\in\mathbb{R}\setminus\mathbb{Q}$
for some choices of $\alpha\in\mathbb{Q}$
(such as for prime $p\in\mathbb{N}$ and $\alpha=-p\in\mathbb{Z}\subset\mathbb{Q}\subset\mathbb{R}).$
Further yet, for real $\alpha>0,$
imaginary solutions $r,s\in\mathbb{C}\setminus\mathbb{R}$
would be needed
to satisfy $-\alpha=r^2.$
The apparent ease of solving the system \eqref{eq:prodsum}
has only been earned through generations
of mathematical discovery and invention.

Yet systems like this are accessible enough to have earned their place
in the collective memory of human history,
having been preserved over the ages in cuneiform and other mediums,
but also entering into the collection of math puzzles
that are likely to be encountered at some point
in the pre-algebra era of primary education.

Eves \cite{HEves}(p. 58)
refers to Babylonian mathematics problems
and a Louvre tablet dating to 300 B.C.
essentially presenting solutions
to two-variable systems
given a known product and sum.
Such systems could be motivated from
geometric interpretations of known area and semi-perimeter.
Eves (p. 87) continues later in discussing
an method attributed to the Pythagoreans in the text section
``Geometric Solution of Quadratic Equations''
in which applications of areas and
appealing to propositions from Euclid's \emph{Elements}
results in a satisfying melding of the geometric and algebraic problems.
The long standing history of connecting systems such as \eqref{eq:prodsum}
and their connection to the study of quadratics and geometric modeling
is clear based on extant mathematical sources.

Solving problems equivalent to solving a quadratic equation
have inspired varied approaches through the history of algebra,
but the nature of products and the connection to areas in basic geometry
make it unsurprising that many approaches take a geometric perspective.
However, the perspective we take
in solving quadratic equations
is informed from development
of function transformation theory,
as we see in \ref{sec:trans:parabola},
and the subsequent symmetry arguments
to which we previously alluded.

\section{A Symmetry Motivated Proof of the Quadratic Formula}
\label{sec:symproof}

\begin{theorem}
	Let $b,c\in\mathbb{R}$ and nonzero $a\in\mathbb{R}$ yield
	real quadratic polynomial function
	\begin{align}\label{eq:gen}
		f(x)&=a\cdot x^2+b\cdot x+c.
	\end{align}
	Then there exists real $h,k\in\mathbb{R}$
	and complex $r,s\in\mathbb{C}$
	such that $f$ has standard form
	\begin{align}\label{eq:stan}
		f(x)&=a\cdot(x-h)^2+k
	\end{align}
	and a linear factored form
	\begin{align}\label{eq:fac}
		f(x)&=a\cdot (x-r)(x-s),
	\end{align}
	where constructively
	\begin{align*}
		h & =-\frac{b}{2a},
		&k & = \frac{4\cdot a\cdot c-b^2}{4a},
		& r,s&=-\frac{b}{2a}\pm \frac{\sqrt{b^2-4\cdot a\cdot c}}{2a}.
	\end{align*}
\end{theorem}
\begin{proof}
	Let $a,b,c\in\mathbb{R}$ with $a\neq 0,$ and $f$ be given by $f(x)=a\cdot x^2+b\cdot x+c.$

	Two polynomial functions of equal degree are equivalent
	precisely when every coefficient of their general form correspond and are numerically equal.
	Through normalization of \eqref{eq:gen},\eqref{eq:stan}, and \eqref{eq:fac},
	and comparison of equivalent monic quadratic polynomials,
	it is clear the lead coefficient ($a$) across any equivalent forms must be equal.
	We proceed under this conclusion.

	For any transformed parabolic function graph given by standard form \eqref{eq:stan},
	if it is presumed equivalent to a factored form \eqref{eq:fac}
	and general form \eqref{eq:gen},
	then we have three cases for existence of $x-$intercepts:
	\begin{enumerate}[label=\bfseries \roman*.)]
		\item The graph of $y=f(x)$ has distinct $x$-intercepts at $(r,0)$ and $(s,0)$
			that are equidistant from the axis of symmetry $x=h,$
		\item the graph of $y=f(x)$ has a single $x$-intercept at the vertex with $(r,0)=(h,k),$ or
		\item the graph of $y=f(x)$ has no real $x$-intercepts, and $r,s\in\mathbb{C}\setminus\mathbb{R}.$ 
	\end{enumerate}
	
	From our cases above, it is clear for such $h\in\mathbb{R}$ and $r,s\in\mathbb{C}$ to exist,
	it must hold by symmetry of the transformed parabola that
	whenever the graph $y=f(x)$ has precisely one or two $x-$intercepts, then necessarily
	\[ h=\frac{r+s}{2}.\]
 	When there are no $x-$intercepts and $r,s\notin \mathbb{R}$,
 	then $r,s$ must form a complex conjugate pair with $s=\overline{r},$
 	as expansion of presumed equivalent forms \eqref{eq:stan} and \eqref{eq:fac}
 	into a general form given by \eqref{eq:gen}
 	requires the coefficient of $x$ in each expansion equates
	\begin{align}\label{eq:xcoeff}
		b & = -a\cdot 2h = a\cdot (-r-s),
	\end{align}
	and the average of the roots $r,s$
	must be equal to the horizontal shifting $h,$
	as previously concluded in the distinct intercept case.
	
	Based on \eqref{eq:xcoeff},
	if a standard form for a general quadratic \eqref{eq:gen} exists,
	then it must hold that the axis of symmetry $x=h$ satisfies
	\begin{align}\label{eq:axis}
		h & = -\frac{b}{2a}.
	\end{align}
	Furthermore, as (complex) roots $r,s$ must have real average $h,$
	and without loss of generality assuming $s>r$ when they are real and distinct,
	we may write both roots in a (real or complex) symmetric form
	\begin{align}\label{eq:symmz}
		r & =h-z,	
		&
		s & =h+z,
	\end{align}
	where $z\in [0,\infty)\cup i(0,\infty),$
	depending on whether we encounter the cases of
	\begin{enumerate}[label=\bfseries \roman*.)]
		\item distinct $x$-intercepts at $(r,0)$ and $(s,0),$ where $z=\frac{s-r}{2}\in (0,\infty),$ 
		\item a single $x$-intercept, where $z=0,$ or 
		\item no real $x$-intercepts, where $z=i\xi$ with $\xi\in(0,\infty).$
	\end{enumerate}
	The symmetric form of \eqref{eq:symmz} then allows us to write
	the product $r\cdot s$ as a difference of squares:
	\begin{align}\label{eq:rsprodz}
		r\cdot s & =h^2-z^2.
	\end{align}
	
	At this point, we have exhausted the extent to which 
	the axis-of-symmetry and linear coefficient analysis can produce results readily,
	and we begin consideration of the quadratic function's extremum ($k$) at the parabola's vertex.
	Returning to expansions of each form,
	with attention given now to the constant term,
	and by subsequently by using \eqref{eq:rsprodz},
	we conclude that the following must hold:
	\begin{align}\label{eq:constant}
		c  = a\cdot h^2+k &= a\cdot (r\cdot s) \\
		& = a\cdot (h^2-z^2). \nonumber
	\end{align}
	
	It is also at this point that we see a nonlinear system for $r,s$
	based on parameters $a,b,c$
	arise based on the correspondence of constant terms in \eqref{eq:constant}
	and the linear coefficient terms in \eqref{eq:xcoeff}.
	Respectively isolating the product and sum of proposed roots $r,s$
	in these results yields 	a classic system
	known as an example of Vieta's formulas \cite{cao2020quadratic}
	(as formulaic representations of polynomial coefficients with respect to roots).
	The resulting system
	\begin{equation} \label{eq:abcprodsum}
		\begin{cases}
			\hfill\frac{c}{a} &=r\cdot s
			\\ \\
			-\frac{b}{a} &= r+s
		\end{cases},
	\end{equation}
	is of the form discussed in \ref{sec:system}.
	For the purpose of identification of quadratic roots
	or $x-$intercepts of a parabola,
	the equivalence of using the quadratic formula
	to solving this system should be recognized.
	In fact, it is almost this very system which is under consideration
	when the algebraic task of quadratic factorization
	is posed in the typical algebra course.
	Following normalization to a monic quadratic,
	the imperative to find two numbers
	with a given sum and product is familiar,
	though rarely applied for complex solutions.
	
	That said, our purposes go beyond merely deriving quadratic roots,
	but also include relating said roots to the parabola's vertex
	when written in standard form,
	so we continue with those goals in mind.
	
	By isolating the $z^2$ term in \eqref{eq:constant},
	a substitution of the axis-of-symmetry result \eqref{eq:axis} yields
	\begin{align}\label{eq:z2}
		z^2 & = \frac{a\cdot h^2-c}{a}
			= \frac{a\cdot \left(\frac{-b}{2a}\right)^2-c}{a}
			= \frac{b^2-4\cdot a\cdot c}{4a^2}.
	\end{align}
	At this point, since $a,b,c\in\mathbb{R}$ by assumption,
	we have $z^2\in\mathbb{R}$ necessarily.
	However, it is precisely the numerator in \eqref{eq:z2}
	(that is, traditionally referred to as the discriminant $\Delta=b^2-4\cdot a\cdot c$)
	which distinguishes between our cases of:
	\begin{enumerate}[label=\bfseries \roman*.)]
		\item distinct $x-$intercepts as in Figures \ref{fig:disc:anegdelpos} \ref{fig:disc:aposdelpos}
			and distinct real roots $r,s=h\pm z = -\frac{b}{2a}\pm \frac{\sqrt{b^2-4\cdot a\cdot c}}{2a},$
						
		\item and a single $x-$intercept as in Figures \ref{fig:disc:anegdelzero} \ref{fig:disc:aposdelzero}
			 and a double real root $r=s=h\pm z = -\frac{b}{2a}\pm \frac{\sqrt{0}}{2a} = -\frac{b}{2a},$
			
		\item or no $x-$intercepts) as in Figures \ref{fig:disc:anegdelneg} \ref{fig:disc:aposdelneg}
			and non-real conjugate roots  $r,s=h\pm z = -\frac{b}{2a}\pm i\cdot \frac{\sqrt{4\cdot a\cdot c-b^2}}{2a},$	
	\end{enumerate}
	
	\begin{figure}[h]\centering
			\begin{subfigure}{.3\textwidth}
				  \includegraphics
				  	{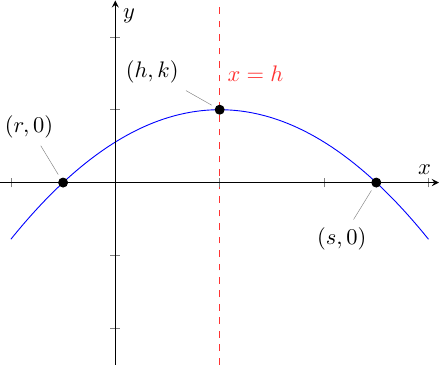}
				  \caption{$a<0,\Delta>0$}
				  \label{fig:disc:anegdelpos}
			\end{subfigure}
			\begin{subfigure}{.3\textwidth}
				  \includegraphics
				  	{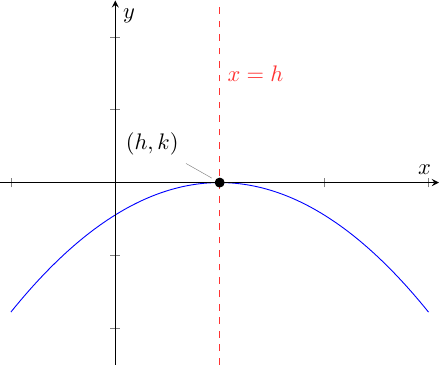}
				  \caption{$a<0,\Delta=0$}
				  \label{fig:disc:anegdelzero}
			\end{subfigure}
			\begin{subfigure}{.3\textwidth}
				  \includegraphics
				  	{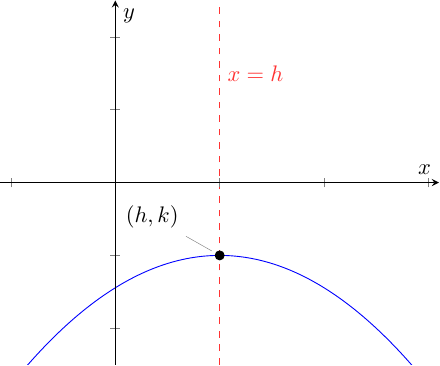}
				  \caption{$a<0,\Delta<0$}
				  \label{fig:disc:anegdelneg}
			\end{subfigure}
			
			\vspace{0.2in}
			
			\begin{subfigure}{.3\textwidth}
				  \includegraphics
				  	{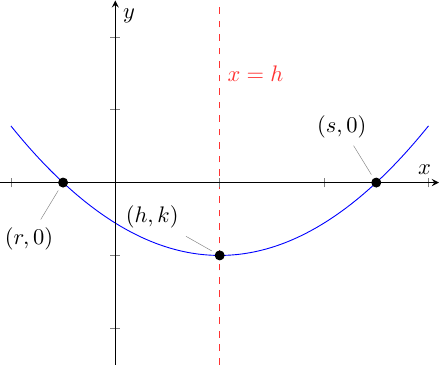}
				  \caption{$a>0,\Delta>0$}
				  \label{fig:disc:aposdelpos}
			\end{subfigure}
			\begin{subfigure}{.3\textwidth}
				  \includegraphics
				  	{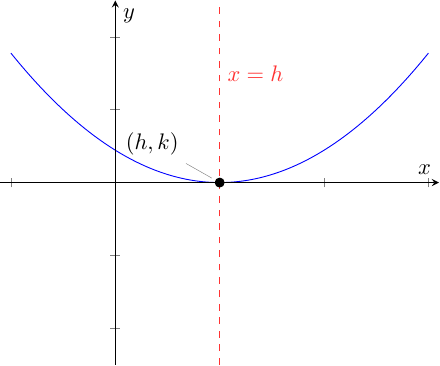}
				  \caption{$a>0,\Delta=0$}
				  \label{fig:disc:aposdelzero}
			\end{subfigure}
			\begin{subfigure}{.3\textwidth}
				  \includegraphics
				  	{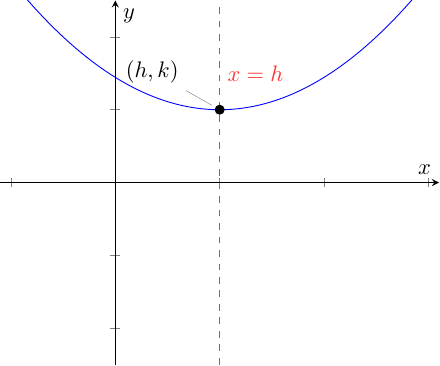}
				  \caption{$a>0,\Delta<0$}
				  \label{fig:disc:aposdelneg}
			\end{subfigure}
			
			\caption{Discriminant $\Delta=b^2-4\cdot a\cdot c,$ cases $\Delta >0,\Delta=0,\Delta<0$}
			\label{fig:disc}
			\end{figure}
	
	Up to case-based insight into the nature of solutions $z$ to \eqref{eq:z2},
	where purely imaginary solutions arise when discriminant $b^2-4\cdot a\cdot c<0$
	and we are required to interpret square roots of a negative number,
	we see that all cases follow the same formula - the quadratic formula:
	\[r,s =  -\frac{b}{2a}\pm \frac{\sqrt{b^2-4\cdot a\cdot c}}{2a},\]
	where we choose to retain the emphasis of the symmetry
	about axis of symmetry $x=h=-\frac{b}{2a}.$
	
	Having formulaically constructed $h,r,s$ using parameters $a,b,c,$
	all that remains to finish our simultaneous constructive proof
	of both factored and standard forms of quadratics
	is to construct extremum $k.$
	Rather than the typical approach of merely evaluating
	$k=f\left(-\frac{b}{2a}\right)$ in the general form of $f,$
	we instead again appeal to coefficient correspondence of terms in our expansions.
	By \eqref{eq:constant},
	it is clear that we may solve for
	\begin{align*}
		k 	 = c-a\cdot h^2
			& = -a\cdot z^2 \\
			& = -a\cdot \left(\frac{b^2-4\cdot a\cdot c}{4a^2}\right) = \frac{4\cdot a\cdot c-b^2}{4a},
	\end{align*}
	where $a,b,c\in\mathbb{R}$ imply $k\in\mathbb{R}.$
\end{proof}

\begin{coro}[The Quadratic Formula]
	Let $b,c\in\mathbb{R}$ and non-zero $a\in\mathbb{R}.$
	Then the quadratic equation
	\begin{align}\label{eq:realquad}
		0&=a\cdot x^2 +b\cdot x + c
	\end{align}
	has solutions in $\mathbb{C}$ of
	\[ x=-\frac{b}{2a}\pm \frac{\sqrt{b^2-4\cdot a\cdot c}}{2a}.\]
\end{coro}
\begin{proof}
	By the Zero Product Property of $\mathbb{C},$
	roots $r,s$ in the (complex) factorization given in \eqref{eq:fac}
	satisfy \eqref{eq:realquad}.
\end{proof}
Furthermore, as coefficients
$a,b,c\in\mathbb{R}\subset\mathbb{C},$
a field, the polynomial ring $\mathbb{C}[x]$
is a Euclidean Domain (with Euclidean valuation of polynomial degree)
which admits a division algorithm,
and is therefore a Unique Factorization Domain.  (See Judson \cite{TJudson} or \cite{Judson}, Chapter 18.)
Hence, the roots $r,s\in\mathbb{C}$ and associated factors are unique up to reordering.

\section{Comments on Quadratic Formula Extensions and Limitations}
\label{sec:ext}

As much as the historical study of quadratics
prompted an expansion of the number system from
$\mathbb{Q}$ to $\mathbb{R}$ to $\mathbb{C},$
of interest to many is how quadratic polynomial structures
in other mathematical systems (numerical or otherwise)
may have constructive proofs of zeros (roots) or appropriate analogs.

An advanced undergraduate development of abstract algebra
tends to resolve discussions of quadratic formula extensions
within the context of fields,
and prior to expansion into the Galois-theory results
which prove non-existence
of coefficient based radical formula for roots
for general degree polynomials (i.e. for quintic or greater).

\subsection{Quadratics in polynomial rings over fields}

\begin{prop}
	Let $\mathbb{F}$ be a field.
	Let $b,c\in\mathbb{F}$ and non-zero $a\in\mathbb{F}.$
	Then the quadratic equation
	\begin{align}\label{eq:fieldquad}
		0&=a\cdot x^2 +b\cdot x + c
	\end{align}
	has solutions
	in splitting field $\mathbb{F}(z)\supseteq \mathbb{F},$
	for some $z$ satisfying $z^2 =\frac{b^2-4\cdot a\cdot c}{4a^2}.$
	Solutions $x\in \mathbb{F}(z)$ are then of the form
	\begin{align}\label{eq:fieldroots}
		x&=-\frac{b}{2a}\pm z.
	\end{align}
\end{prop}

Thomas Judson develops the theory of splitting fields
for polynomials of general degree
in \cite{TJudson} (and \cite{Judson}) (Chapter 21),
and we leave the proof of formal theory to such works.

In the case of $\mathbb{F}=\mathbb{R},$
this just restates that either a real quadratic has real roots (with non-negative discriminant)
or a discriminant $\Delta<0$ implies roots
$x\in\mathbb{R}\left(i\cdot \frac{\sqrt{-\Delta}}{2a}\right)=\mathbb{R}(i)=\mathbb{C}.$\

This result is more interesting in a field such as $\mathbb{F}=\mathbb{Q}$
with irreducible quadratics like $x^2-2x-2$
and roots $1\pm \sqrt{3} \in \mathbb{Q}(\sqrt{3})\subset \mathbb{C}.$
Such considerations also remind
of the ancient controversy over existence of irrational numbers
as a foreshadowing of the controversy of developing imaginary numbers.

Also of interest are the polynomial rings generated by finite fields,
such as modular arithmetic fields
$\mathbb{F}=\mathbb{Z}_p \cong \mathbb{Z}\slash \langle p \rangle$
for prime modulus $p.$
In such cases, many authors will take additional care
to clarify the operation of division in the fraction form of \eqref{eq:fieldroots}
may be better understood as modular multiplication
by the multiplicative inverse $\beta\in\mathbb{Z}_p$ to $2a\in\mathbb{Z}_p$
such that $\beta\cdot (2a) \equiv 1\ (mod\ p).$
Though in a prime modulus
this may seem little more than a semantic distinction,
in attempts to extend to quadratic congruences for composite moduli
it may be more troublesome to clearly convey results
without explicitly adapting notation.

\subsection{General Quadratic Congruences}

\begin{prop}\label{prop:modquad}
	Let modulus $n\in\mathbb{N},$
	modular arithmetic ring $R=\mathbb{Z}_n,$
	with $b,c\in R$ and (nonzero) unit $a\in R.$
	Then the quadratic congruence
	\begin{align}\label{eq:modquad}
		0&\equiv a\cdot x^2 +b\cdot x + c\ (mod\ n)
	\end{align}
	has solutions provided $gcd(2a,n)=1,$
	and the discriminant $\Delta \equiv b^2-4\cdot a\cdot c$
	is a quadratic residue modulo $n.$
	The solutions $x\in\mathbb{Z}_n$ to \eqref{eq:modquad} are of the form
	\[ x \equiv -b\cdot (2a)^{-1}+z\ (mod\ n)\]
	where $z\equiv s\cdot (2a)^{-1}$
	for solutions $s\in\mathbb{Z}_n$ to $s^2\equiv \Delta.$
\end{prop}
	
	Karl-Dieter Crisman develops the theory of quadratic residues, congruences, and reciprocity
	in \cite{KDCrisman} (Chapters 16-17),
	which we will not recreate here.
	We have notationally adapted statements of results there
	to align with the symmetry approaches in \ref{sec:symproof}. 

	The validity of modular multiplicative inversion of $2a\in R$
	is dependent on both $2,a$ being units in $R,$
	which is in turn dependent on $n$ being odd and relatively prime to $2a.$
	Furthermore, the restriction on discriminant $\Delta$
	to be a quadratic residue (i.e. a perfect square modulo $n$)
	is analogous to the non-negative discriminant restriction on real quadratics
	when seeking real quadratic roots.

	For the special case of \emph{prime} $n$
	with \emph{cyclic} (multiplicative) group of units $U_n=\langle g\rangle$
	and primitive root (and generator) $g\in\mathbb{Z}_n,$
	the valid quadratic residues are $\langle g^2\rangle=Q_n,$
	with $\Delta \in Q_n$ being necessary.
	However, for a general composite modulus $n,$
	development of Legendre and Jacobi symbol theory
	is needed for a more thorough assessment of Proposition \ref{prop:modquad}.

	Of note in the modular congruence setting
	is the failure of the Zero Product Property in $\mathbb{Z}_n$
	when modulus $n$ is composite
	(and $\mathbb{Z}_n$ is \emph{not} a field nor integral domain).
	This introduces uncertainty (non-uniqueness) in factorization,
	such as in
	\[ x^2 \equiv (x+2)^2\ (mod\ 4),\]
	and results in a quadratic formula for a general composite modulus
	and arbitrary coefficients to be unobtainable,
	with coefficient compatibility conditions being more restrictive
	than the traditional discriminant compatibility conditions.
	
\subsection{Beyond Fields and Commutative Rings}

	Without going into details,
	there is research
	on additional extensions of quadratic formula
	to quadratics over noncommutative division rings
	such as the quaternion ring $\mathbb{H},$
	or for the physics-relevant adaptation
	of the split-quaternions $\mathbb{H}_s$
	(see \cite{cao2020quadratic}).
	However, the complexity of the computational results
	go well beyond what we wish to present here,
	as accessible connections to a perspective of transformation
	have yet to be uncovered.

\newpage




\end{document}